\documentclass[11pt]{article}%
\usepackage{mathrsfs}
\usepackage{bbm}
\usepackage{amsfonts}
\usepackage{amsmath,amssymb}
\usepackage{amsmath}
\usepackage{amssymb}
\usepackage{graphicx}%
\usepackage{shorttoc}
\providecommand{\U}[1]{\protect \rule{.1in}{.1in}}

\newtheorem{theorem}{Theorem}[section]

\newtheorem{definition}[theorem]{Definition}

\newtheorem{lemma}[theorem]{Lemma}

\newtheorem{remark}[theorem]{Remark}

\newenvironment{proof}[1][Proof]{\noindent \textbf{#1.} }{\  $\Box$}
\numberwithin{equation}{section}

\begin{document}

\title{\textbf{Continuous Ergodic Capacities}}
\author{ Yihao Sheng$^*$ and Yongsheng Song\thanks{Academy of Mathematics and Systems Science, Chinese Academy of Sciences, Beijing 100190, China, and
School of Mathematical Sciences, University of Chinese Academy of Sciences, Beijing 100049, China. E-mails: shengyihao@amss.ac.cn (Y. Sheng), yssong@amss.ac.cn (Y. Song).}  }

\date{\today}
\maketitle

\begin{abstract}
    The objective of this paper is to characterize the structure of the set $\Theta$ for a continuous ergodic upper probability $\mathbb{V}=\sup_{P\in\Theta}P$ (Theorem \ref {main result}): 
    \begin{itemize}
  \item[$\cdot$]  $\Theta$ contains a finite number of ergodic probabilities;  
  \item[$\cdot$] Any invariant probability in $\Theta$ is a convex combination of those ergodic ones in $\Theta$;  
  \item[$\cdot$]Any probability in $\Theta$ coincides with an invariant one in $\Theta$ on the invariant $\sigma$-algebra. 
  \end{itemize}
 The last property has already been obtained in \textsl{Cerreia-Vioglio, Maccheroni, and Marinacci} \cite{ergodictheorem},   which firstly studied the ergodicity of such capacities.  
 
 As an application of the characterization, we prove an ergodicity result (Theorem \ref {improve}), which improves the result in \cite{ergodictheorem} in the sense that the limit of the time mean of $\xi$ is bounded by the upper expectation $\sup_{P\in\Theta}E_P[\xi]$, instead of the Choquet integral. Generally, the former is strictly smaller.
\end{abstract}


\textbf{MSC-classification}: 28A12; 37A05.

\section{Introduction}

\quad \quad The notions of invariant and ergodic capacities were initiated by \textsl{Cerreia-Vioglio, Maccheroni, and Marinacci} \cite{ergodictheorem}, in which they gave the first Birkhoff's ergodic theorem for continuous invariant capacities. We state their main results below in the notations of this paper.

\emph{
Let $(\Omega,\mathcal{F}, T)$ be a measurable system, and let $\mathbb{V}=\sup_{P\in\Theta}P$ be a continuous upper probability.  If $\mathbb{V}$ is $T$-invariant, then for any $P\in\Theta$ and any bounded random variable $\xi$ on $\Omega$, one has
\begin{eqnarray}\nonumber
    \begin{split}
        \lim_{n\to\infty}\frac{1}{n}\sum_{k=0}^{n-1}\xi(T^{k}\omega) \ \ \mbox{exists}, \ \ P\mbox{-a.s.}
    \end{split}
\end{eqnarray}
If further assuming that $\mathbb{V}$ is ergodic, then for any $P\in\Theta$, one gets
\begin{eqnarray}\label{ergodicresult1}
    \begin{split}
        -C_{\mathbb{V}}[-\xi^{*}]\leq\lim_{n\to\infty}\frac{1}{n}\sum_{k=0}^{n-1}\xi(T^{k}\omega)\leq C_{\mathbb{V}}[\xi^{*}], \ \ P\mbox{-a.s.},
    \end{split}
\end{eqnarray}
where $\xi^{*}(\omega)=\limsup\limits_{n\to\infty}\frac{1}{n}\sum_{k=0}^{n-1}\xi(T^{k}\omega)$ and $C_{\mathbb{V}}$ represents the Choquet integral w.r.t $\mathbb{V}.$
}

The main purpose of this paper is to provide a characterization of the continuous ergodic capacity $\mathbb{V}$. Specifically, let $\Theta$ be the set of probabilities on $\Omega$ dominated by $\mathbb{V}$,
 and let $\Theta_{0}$ (resp. $\Theta_*$) be the subset of $T$-invariant (resp. ergodic) probabilities in $\Theta$. Then, we get
\begin{itemize}
    \item  The cardinality of $\Theta_*$ is finite;
    \item $\Theta_0=\textmd{co}\ \Theta_*$, the convex hull of $\Theta_*$.
\end{itemize}

As a by-product, we improve the ergodic theorem (\ref{ergodicresult1}). Let $\mathbb{E}^0=\sup_{P\in\Theta_0}E_{P}.$ Then, for any $P\in\Theta$ and any bounded random variable $\xi$, we have
$$
    -\mathbb{E}^0[-\xi]\leq\lim_{n\to\infty}\frac{1}{n}\sum_{k=0}^{n-1}\xi(T^{k}\omega)\leq \mathbb{E}^0[\xi], \ \ P\mbox{-a.s.}
$$
 The limit is bounded by the upper expectation $\mathbb{E}^0[\xi]$, instead of the Choquet integral $C_{\mathbb{V}}[\xi^{*}].$
  
The rest of this paper is organized as follows. In Section 2, we present some notations of the continuous capacities. The main results are presented and proved in Section 3. In section 4, as an appendix, we give a brief introduction to Banach-Mazur limits \cite{infiniteanalysis}.

\section{Some Notations}
\
\
\
\
Let $(\Omega,\mathcal{F})$ be a measurable space, and let $T:\Omega\to\Omega$ be a measurable transformation. Denote by $\mathcal{I}$ the $T$-invariant algebra. Let $\mathbb{V}(A)=\sup_{P\in\Theta}P(A),$ $A\in\mathcal{F},$ be an upper probability, where $\Theta$ is a family of probabilities on $(\Omega,\mathcal{F}).$ Denote by $\mathcal{M} (\Omega)$ the collection of probabilities on $(\Omega,\mathcal{F}).$  In this paper, we assume that \[\Theta=\big\{ P\in \mathcal {M} (\Omega) \ | \ P(A)\le \mathbb{V}(A), \ A\in\mathcal{F}\big\}.\]

Let $\mathbb{E}[\xi]=\sup_{P\in\Theta}E_{P}[\xi],$ where $\xi$ is a $\mathcal{F}$-measurable function with $\sup_{P\in\Theta}E_{P}[|\xi|]<\infty.$ We call $\mathbb{E}$ the upper expectation with respect to $\Theta.$ 

\begin{definition}
    Let $\mathbb{V}(A)=\sup_{P\in\Theta}P(A)$, $A\in\mathcal{F},$ be an upper probability with respect to $\Theta.$ We say $\mathbb{V}$ is continuous if it continuous from above, namely, $\lim\limits_{n\to\infty}\mathbb{V}(A_{n})=0$ for $A_n\in \mathcal{F}$ with $A_{n}\downarrow \emptyset$.
\end{definition}

\begin{remark}\label{continuity of capacity} Let $A_{n}\in\mathcal{F}, n\ge 1$, be a sequence of disjoint sets. If $\mathbb{V}(A_{n})\ge\varepsilon$ for some $\varepsilon>0$ and all $n\ge 1$, then $\mathbb{V}$ cannot be continuous. In fact, $\mathop{\cup}\limits_{k\geq n}A_{k}\downarrow \varnothing,$ but $\mathbb{V}(\mathop{\cup}\limits_{k\geq n}A_{k})\geq\mathbb{V}(A_{n})\ge\varepsilon.$
\end{remark}

\begin{definition}
    Let $\mathbb{V}$ be a continuous upper probability on $(\Omega,\mathcal{F})$. We say $\mathbb{V}$ is $T$-invariant if for any $A\in\mathcal{F},$ we have \[\mathbb{V}(A)=\mathbb{V}(T^{-1}A).\]
\end{definition}

The following lemma is from \cite{ergodictheorem}. For the reader's convenience, we give a brief proof here.

\begin {lemma} \label {Lambda-inva} Assume that $\mathbb{V}(A)=\sup_{P\in\Theta}P(A)$ is $T$-invariant. Let $\Theta_0$ be the subset of $T$-invariant probabilities in $\Theta$. For each $P\in\Theta$, there exists $P'\in\Theta_0$ such that $P=P'$ on the $T$-invariant algebra $\mathcal{I}$.
\end {lemma}
\begin {proof}
For $P\in\Theta$, set $P_n(A)=\frac{1}{n}\sum_{k=0}^{n-1}P(T^{-k}A)$. Let $\Lambda$ be a Banach–Mazur limit, (see Definition \ref{banachmazur}) and set $P'(A)=\Lambda\big((P_n(A))\big)$, which is additive since $\Lambda$ is a linear functional. It follows from the positivity of $\Lambda$ that $P'(A)\le\Lambda\big((\mathbb{V}(A))\big)=\mathbb{V}(A)$. The continuity of $\mathbb{V}$ implies that $P'$ is a probability. So we get $P\in\Theta$. Noting that $P_n(A)=P(A)$ for any $A\in\mathcal{I}$ and $n\in\mathbb{N}$, we get $P'(A)=P(A)$ for each $A\in\mathcal{I}$. For any $A\in\mathcal{F}$, note  that $\limsup_{n}|P_n(A)-P_n(T^{-1}A)|=0$.  By Remark \ref {Rem-banachmazur}, we have\[P'(T^{-1}A)=\Lambda\big((P_n(T^{-1}A))\big)=\Lambda\big((P_n(A))\big)=P'(A),\] i.e., $P'$ is $T$-invariant.
\end {proof}

\begin{definition}
    Let $\mathbb{V}$ be a continuous upper probability on $(\Omega,\mathcal{F}),$ We say $\mathbb{V}$ is $T$-ergodic, if $\mathbb{V}(A)=0,$ or 1, for $A\in\mathcal{I}.$
\end{definition}

\section{Characterizations of Ergodic Upper Probabilities}

In this section, we establish a characterization for continuous ergodic upper probabilities (Theorem \ref {main result}).  As an application, we prove an ergodicity result (Theorem \ref {improve}), which improves the result in \cite{ergodictheorem}.

\begin {lemma} \label {0-1} Assume that $\mathbb{V}(A)=\sup_{P\in\Theta}P(A)$ is $T$-ergodic. Let $\Theta_0$ be the subset of $T$-invariant probabilities in $\Theta$. For $A\in \mathcal{I}$ with $\mathbb{V}(A)>0$, there exists $P'\in\Theta_0$ such that $P'(A)=1$.
\end {lemma}
\begin {proof}
For $A\in \mathcal{I}$ with $\mathbb{V}(A)>0$, there exists $P_n\in\Theta$ such that $P_n(A)\rightarrow1$. Let $\Lambda$ be a Banach–Mazur limit, and set $P(B)=\Lambda\big((P_n(B))\big), B\in\mathcal{F}$. By the properties of $\Lambda$ and the continuity of $\mathbb{V}$, we get that $P$ is a probability belonging to $\Theta$ with $P(A)=\lim_n P_n(A)=1$. It follows from Lemma \ref {Lambda-inva} that there exists $P'\in\Theta_0$ such that $P'(A)=1$.
\end {proof}

\begin {lemma}\label {ABC} Assume that $\mathbb{V}(A)=\sup_{P\in\Theta}P(A)$ is $T$-ergodic.   Let $\Theta_0$ be the subset of $T$-invariant probabilities in $\Theta$. For a $T$-ergodic probability $P^*$, if there exists $P\in\Theta_0$ such that $P^*$ is absolutely continuous with respect to $P$, we have $P^*\in \Theta_0$.
\end {lemma}
\begin {proof} For a $T$-ergodic probability $P^*$, set $\Theta_{P^*}=\{ P\in \Theta_0 \ | \ P^*\ll P\}$. Assume that $\Theta_{P^*}$ is not empty, but does not contain $P^*$. 

Two invariant probability measures coincide if and only if they are equal on the invariant $\sigma$-algebra $\mathcal{I}$. In the sequel, we shall consider the measurable space $(\Omega, \mathcal{I})$.  

Set $\beta_0=\sup_{P\in\Theta_0}(P^*-P)^+(\Omega)$, which is positive by the assumption $P^*\not\in \Theta_0$. Choose $P_0\in\Theta_0$ such that $(P^*-P_0)^+(\Omega)\ge \frac{1}{2}\beta_0$. Let $\Omega=D_0\cup D_0^c$ be the Hahn decomposition of the signed measure $P^*-P_0$ with \[(P^*-P_0)(D_0)=(P^*-P_0)^+(\Omega).\]
So we get $P^*(D_0)=1$ and $P_0(D_0^c)\ge \frac{1}{2}\beta_0$.

Let $\Theta_{1}=\{ P\in\Theta_0 \ | \ P(D_0)=1 \}$. Since $P^*(D_0)=1$ and $\Theta_{P^*}$ is not empty, there exists $P'\in\Theta_0$ such that $P'(D_0)>0$.  By Lemma \ref {0-1} we get that $\Theta_1$ is nonempty. Set $\beta_1=\sup_{P\in\Theta_1}(P^*-P)^+(\Omega)$, which is positive. Choose $P_1\in\Theta_1$ such that $(P^*-P_1)^+(\Omega)\ge \frac{1}{2}\beta_1$. Let $\Omega=D_1\cup D_1^c$ be the Hahn decomposition of $P^*-P_1$ with $D_1\subset D_0$ and \[(P^*-P_1)(D_1)=(P^*-P_1)^+(\Omega).\]
Then, we get $P^*(D_1)=1$ and $P_1(D_1^c\cap D_0)\ge \frac{1}{2}\beta_1$.

By induction, we can obtain $(\Theta_k, \beta_k, P_k, D_k)_{k\ge0}$ satisfying
\begin {eqnarray*}
& &\Theta_{k+1}=\{ P\in\Theta_k \ | \ P(D_k)=1 \} \neq\emptyset,\\
& &\beta_{k+1}=\sup_{P\in\Theta_{k+1}}(P^*-P)^+(\Omega)>0,\\
& &P_{k+1}\in\Theta_{k+1} \ \textit{with} \ (P^*-P_{k+1})^+(\Omega)\ge \frac{1}{2}\beta_{k+1},\\
& &\Omega=D_{k+1}\cup D_{k+1}^c \ \textit{being the Hahn decomposition of} \ P^*-P_{k+1} \ \textit{with} \\ 
& &D_{k+1}\subset D_k, \ P^*(D_{k+1})=1 \ \textit{and} \ P_{k+1}(D_{k+1}^c\cap D_k)\ge \frac{1}{2}\beta_{k+1}.
\end {eqnarray*}
Noting that $\{D_{k+1}^c\cap D_k\}_{k\ge0}$ is disjoint, and $\mathbb{V}(D_{k+1}^c\cap D_k)\ge P_{k+1}(D_{k+1}^c\cap D_k)\ge \frac{1}{2}\beta_{k+1}>0.$ Then $\mathbb{V}(D_{k+1}^c\cap D_k)=1,$ which is a contradiction by Remark \ref{continuity of capacity}.
\end {proof}

Now, we provide the following characterization of the continuous ergodic capacity $\mathbb{V}.$

\begin {theorem}\label{main result}Assume that $\mathbb{V}(A)=\sup_{P\in\Theta}P(A)$ is $T$-ergodic.  Let $\Theta_0$ (resp. $\Theta_*$) be the subset of $T$-invariant (resp. ergodic) probabilities in $\Theta$. Then, the cardinality of $\Theta_*$ is finite, and $\Theta_0=\emph{co}\ \Theta_*$, the convex hull of $\Theta_*$.
\end {theorem}

\begin {proof}The finiteness of the cardinality of $\Theta_*$ follows directly from the continuity of $\mathbb{V}$ and Remark \ref{continuity of capacity}.  Now let us prove  $\Theta_0=\textmd{co} \ \Theta_*$. For $P\in\Theta_0$ which is ergodic, we get the desired result. Otherwise, we can choose a set $A\in\mathcal{I}$ such that $0<P(A)<1$. Set $A_0=A$, $A_1=A^c$, which is called a nontrivial decomposition of $\Omega$. If $A_{\alpha}, \alpha\in \{0, 1\}$ does not have a nontrivial decomposition, the procedure halts. Otherwise, we have a decomposition $A_{\alpha}=A_{\alpha0}+A_{\alpha1}$ for $A_{\alpha0}, A_{\alpha1}\in \mathcal{I}$ with $0<P(A_{\alpha0}), P(A_{\alpha1})<P(A_{\alpha})$. Then, we repeat the above procedure for $A_{\alpha}, \alpha\in \{0,1\}^n$, $n\ge2$.

We claim that the above procedure will halt after a finite number of steps. Otherwise, we shall get a sequence of disjoint sets $(B_n)\subset \mathcal{I}$ with $P(B_n)>0$, $n\ge1$. By the ergodicity of $\mathbb{V}$, we have $\mathbb{V}(B_n)=1$, $n\ge1$, which is impossible since $\mathbb{V}$ is continuous.

Therefore, we get a finite partition $\{B_k\}_{k=1}^n$ of $\Omega$ with $B_k\in\mathcal{I}$, $P(B_k)>0$, and none of  $B_k$ having a nontrivial decomposition. Then, $P_k(\cdot)=P(\cdot | B_k)$ is ergodic and absolutely continuous with respect to $P$, so we have $P_k\in\Theta_*$ by Lemma \ref {ABC}. Noting that \[P(C)=\sum_{k=1}^nP(B_k)P_k(C), \ \textmd{for} \ C\in\mathcal{F}, \] we get the desired result.
\end {proof}

\vspace{0.5cm}

Based on the above characterization, we can improve the ergodic theorem (\ref{ergodicresult1}) as follows:

\begin{theorem}\label{improve}
    Assume that $\mathbb{V}(A)=\sup_{P\in\Theta}P(A)$ is $T$-ergodic. Let $\Theta_0$ be the subset of $T$-invariant probabilities in $\Theta$. Then for any $P\in\Theta$ and any bounded random variable $\xi$, we have
    $$
            -\mathbb{E}^0[-\xi]\leq\lim_{n\to\infty}\frac{1}{n}\sum_{k=0}^{n-1}\xi(T^{k}\omega)\leq \mathbb{E}^0[\xi], \ \ P\mbox{-a.s.},
    $$ where $\mathbb{E}^0$ is the upper expectation with respect to $\Theta_0$.
\end{theorem}

\begin{proof} Since $\xi^*(\omega)=\limsup\limits_{n\to\infty}\frac{1}{n}\sum_{k=0}^{n-1}\xi(T^{k}(\omega))$ and $\xi_*(\omega)=\liminf\limits_{n\to\infty}\frac{1}{n}\sum_{k=0}^{n-1}\xi(T^{k}(\omega))$ are both $\mathcal{I}$-measurable, by Lemma \ref {Lambda-inva}, it suffices to prove that, for any $P\in\Theta_0$, \[P\big\{-\mathbb{E}^0[-\xi]\le \xi_*=\xi^*\le \mathbb{E}^0[\xi]\big\}=1.\]
Let $\Theta_*=\{P_{i}\}_{i=1}^{m}$ and let $\Omega=D_1\cup\cdots\cup D_m$ be the $\mathcal{I}$-measurable partition such that $P_{i}(D_{i})=1,$ $i=1,\cdots,m$. By Theorem \ref{main result}, we have $E_{P}[\xi|\mathcal{I}]=\sum_{i=1}^{m}E_{P_{i}}[\xi]1_{D_{i}},$ $P$-a.s. Then, Birkhoff's ergodic theorem gives
    $$
      \lim_{n\to\infty}\frac{1}{n}\sum_{k=0}^{n-1}\xi(T^{k}\omega)=\sum_{i=1}^{m}E_{P_{i}}[\xi]1_{D_{i}}, \ \ P\mbox{-a.s.},
    $$
which implies the desired result.
\end{proof}

\begin {remark}
The above theorem improves the ergodic theorem in \cite{ergodictheorem} in two aspects.  Firstly, the upper expectation $\mathbb{E}^0[\xi]$ is strictly smaller  than the Choquet  integral $C_{\mathbb{V}}[\xi^{*}]$ generally. Secondly, for the general case, no explicit form of $\xi^*(\omega)=\limsup\limits_{n\to\infty}\frac{1}{n}\sum_{k=0}^{n-1}\xi(T^{k}(\omega))$ was given in \cite{ergodictheorem}.
\end {remark}

\section{Appendix}
\
\
\
\
Banach-Mazur limits (see \emph{Infinite Dimensional Analysis} by Aliprantis and Border, 2005 \cite{infiniteanalysis}) play an important role in the study of invariant capacities. Therefore, we shall give a brief introduction to it.

\begin{definition}\label{banachmazur}
    Let $\ell_{\infty}$ be the space of all bounded sequences. A positive linear functional $\Lambda: \ell_{\infty} \rightarrow \mathbb{R}$ is a Banach-Mazur limit if
    \begin{enumerate}
        \item $\quad \Lambda(\boldsymbol{e})=1$, where $\boldsymbol{e}=(1,1,1, \ldots)$,
        \item $\quad \Lambda\left(x_1, x_2, \ldots\right)=\Lambda\left(x_2, x_3, \ldots\right)$ for each $\left(x_1, x_2, \ldots\right) \in \ell_{\infty}$.
    \end{enumerate}
\end{definition}

\begin{lemma}
    If $\Lambda$ is a Banach-Mazur limit, then
$$
\liminf _{n \rightarrow \infty} x_n \leqslant \Lambda(x) \leqslant \limsup _{n \rightarrow \infty} x_n
$$
for each $x=\left(x_1, x_2, \ldots\right) \in \ell_{\infty}$. In particular, $\Lambda(x)=\lim\limits_{n \rightarrow \infty} x_n$ for each convergent sequence $x$ (so every Banach-Mazur limit is an extension of the limit functional).
\end{lemma}

\begin {remark} \label {Rem-banachmazur}For any $x, y\in \ell_{\infty}$,
\[|\Lambda(x)-\Lambda(y)|\le\lim_n\sup_{k\ge n} |x_k-y_k|.\]
\end {remark}

\begin{theorem}
    Banach–Mazur limits exist.
\end{theorem}

\renewcommand{\refname}{\large References}{\normalsize \ }

\end{document}